\documentclass[11pt]{amsart}

\usepackage{graphicx,color}
\usepackage{latexsym}
\usepackage{graphicx} 
\usepackage{amsthm,amsmath, amssymb,amsfonts}

\theoremstyle{plain}
\newtheorem{theorem}{Theorem}

\newtheorem{proposition}[theorem]{Proposition}

\theoremstyle{remark}

\theoremstyle{definition}

\title[Rigidity of area-minimizing free boundary surfaces]{Rigidity of area-minimizing free boundary surfaces in mean convex three-manifolds}
\author{Lucas C. Ambrozio}
\address{Instituto de Matem\'atica Pura e Aplicada (IMPA) \\ Estrada Dona Castorina 110 \\ 22460-320 Rio de Janeiro \\ Brazil; lambroz@impa.br}
\thanks{The author was supported by CNPq-Brazil and FAPERJ}

\begin{document}

\begin{abstract}
  We prove a local splitting theorem for three-manifolds with mean convex boundary
  and scalar curvature bounded from below that contain certain locally area-minimizing free boundary surfaces.
  Our methods are based on those of Micallef and Moraru \cite{MM}. We use this local result to establish
  a global rigidity theorem for area-minimizing free boundary disks. In the negative scalar curvature case,
  this global result implies a rigidity theorem for solutions of the Plateau problem with length-minimizing boundary.
\end{abstract}

\maketitle

\section{Introduction and statements of the results}

\indent Let $M$ be a Riemannian manifold with boundary $\partial M$. Free boundary minimal submanifolds arise as critical
points of the area functional when one restricts to variations that preserve $\partial M$ (but not ne\-cessarily leave
it fixed). Many beautiful known results about closed minimal surfaces could guide the formulation
of analogous interesting questions about free boundary minimal surfaces.
In this paper, inspired by the rigidity theorems for area-minimizing closed surfaces proved in \cite{BBN}, \cite{CG}, \cite{MM} and \cite{N},
we investigate rigidity of area-minimizing free boundary surfaces in Riemannian three-manifolds. \\
\indent Schoen and Yau, in their celebrated joint work, discovered interesting relations between the scalar curvature
of a three-dimensional manifold and the topology of stable minimal surfaces inside it, which emerge when one uses the second
variation formula for the area, the Gauss equation and the Gauss-Bonnet theorem. An example is given 
by the following
\begin{theorem}[Schoen and Yau] \label{SchYau1}
  Let $M$ be an oriented Riemannian three-manifold with positive scalar curvature. Then $M$ has
  no immersed orientable closed stable minimal surface of positive genus.
\end{theorem}
\indent Schoen and Yau used this to prove that any Riemannian metric with non-negative scalar curvature on the three-torus
must be flat. More generally, they proved the following theorem (see \cite{SY}):
\begin{theorem} [Schoen and Yau] \label{SchYau2}
  Let $M$ be a closed oriented three-manifold. If the fundamental group of $M$ contains a subgroup isomorphic
  to the fundamental group of the two-torus, then any Riemannian metric on $M$
  with nonnegative scalar curvature must be flat.
\end{theorem} 
\indent The hypothesis on the fundamental group implies that there exists a continuous map $f$ from the two-torus to $M$
that induces an injective homomorphism $f_{*}$ on the fundamental groups. Then the idea is to apply a minimization procedure
among maps that induce the same homomorphism $f_{*}$ in order to obtain an immersed stable minimal two-torus in $(M,g)$
for any Riemannian metric $g$. Since any non-flat Riemannian metric with nonnegative scalar curvature on a closed
three-manifold can be deformed to a metric with positive scalar curvature (see \cite{KW}), the theorem follows. \\
\indent In \cite{FCS}, Fischer-Colbrie and Schoen observed that an immersed, two-sided, stable minimal two-torus in a Riemannian
three-manifold with nonnegative scalar curvature must be flat and totally geodesic, and conjectured that Theorem \ref{SchYau2} would
hold if one merely assume the existence of an area-minimizing two-torus. This conjecture was established by Cai and Galloway \cite{CG}.
More precisely, they proved that if $M$ is a closed Riemannian three-manifold which contains a two-sided embedded two-torus
that minimizes the area in its isotopy class, then $M$ is flat. The fundamental step was the following local result:
\begin{theorem}[Cai and Galloway] \label{CaiGal}
  If a Riemannian three-manifold with nonnegative scalar curvature contains an embedded, two-sided, locally area-minimizing
  two-torus $\Sigma$, then the metric is flat in some neighborhood of $\Sigma$.
\end{theorem}
\indent In recent years, some similar results were proven for closed surfaces other than tori under different scalar curvature hypotheses.
In particular, we mention the theorems of Bray, Brendle and Neves \cite{BBN} and Nunes \cite{N}.
\begin{theorem}[Bray, Brendle and Neves] \label{BraBreNev}
  Let $(M, g)$ be a three-manifold with scalar curvature greater than or equal to $2$. If $\Sigma$ is an embedded two-sphere that is
  locally area-minimizing, then $\Sigma$ has area less than or equal to $4\pi$. Moreover, if equality holds, then $\Sigma$ with the
  induced metric $g_\Sigma$ has constant Gaussian curvature equal to $1$ and there is a neighborhood of $\Sigma$ in $M$ that is isometric to 
  $((-\epsilon,\epsilon)\times \Sigma, dt^2+g_{\Sigma})$.
\end{theorem}
\begin{theorem}[Nunes] \label{Nun}
  Let $(M, g)$ be a three-manifold with scalar curvature greater than or equal to $-2$. If $\Sigma$ is an embedded, two-sided,
  locally area-minimizing closed surface with genus $g(\Sigma)$ greater than $1$, then $\Sigma$ has area greater than or equal to 
  $4\pi(g(\Sigma)-1)$. Moreover, if equality holds, then $\Sigma$ with the induced metric $g_\Sigma$ has constant Gaussian curvature equal
  to $-1$ and there is a neighborhood of $\Sigma$ in $M$ that is isometric to $((-\epsilon,\epsilon)\times \Sigma, dt^2+g_{\Sigma})$.
\end{theorem}
\indent These local splitting theorems also imply interesting global theorems (see \cite{BBN} and \cite{N}). \\
\indent Let us give a sketch of the proof of Theorems \ref{BraBreNev} and \ref{Nun}. In order to prove the inequalities for the area of the
respective $\Sigma$ in the statements above, one can follow Schoen and Yau, using the stability of $\Sigma$, 
the Gauss equation and the Gauss-Bonnet theorem. These inequalities also appeared in the work of Shen and Zhu \cite{SZ}. When the area of $\Sigma$ achieves the equality stated in the
respective theorems, there are more restrictions on the intrinsic and extrinsic geometries of $\Sigma$
(recall Fischer-Colbrie and Schoen remark), which allowed then to construct a foliation of $M$ around
$\Sigma$ by constant mean curvature surfaces (by using the implicit function theorem). The use of foliations by
constant mean curvature surfaces in relation to scalar curvature problems
has already appeared in the work of Huisken and Yau \cite{HY} and Bray \cite{Bra}. After this point, they prove
that the leaves of the foliation have area not greater than that of $\Sigma$. This is achieved by very different
means in \cite{BBN} and \cite{N}. Since $\Sigma$ is area-minimizing, it follows that each leaf is area-minimizing and its
area satisfies the equality stated in the respective theorems, an information that can be used to conclude the local
splitting of $(M,g)$ around $\Sigma$. \\
\indent An interesting unified approach to Theorems \ref{CaiGal}, \ref{BraBreNev} and \ref{Nun} was provided by Micallef and Moraru
\cite{MM}, also based on foliations by constant mean curvature surfaces. In our paper, we prove an analogous local rigidity
theorem for free boundary surfaces, based on their methods. \\
\indent Our setting is the following. Let $(M,g)$ be a Riemannian three-manifold with boundary $\partial M$. Let $R^M$ denote
the scalar curvature of $M$ and $H^{\partial M}$ denote the mean curvature of $\partial M$ (we follow the convention that a unit
sphere in $\mathbb{R}^3$ has mean curvature 2 with respect to the outward normal). Let $\Sigma$ be a
compact, connected surface with boundary $\partial \Sigma$. We say that $\Sigma$ is properly embedded (or immersed) in $M$
if it is embedded (or immersed) in $M$ and $\Sigma\cap\partial M=\partial \Sigma$.
We say that such $\Sigma$ is locally area-minimizing in $M$ if every nearby properly immersed surface has area greater than
or equal to the area of $\Sigma$. The first variation formula for the area (see the Appendix) implies that an area-minimizing properly immersed surface $\Sigma$
is minimal and \textit{free boundary}, i.e., $\Sigma$ meets $\partial M$ orthogonally along $\partial \Sigma$. Furthermore
$\Sigma$ is \textit{free boundary stable}, i.e., the second variation of area is nonnegative
for every variation that preserves the boundary $\partial M$. \\
\indent When $R^M$ and $H^{\partial M}$ are bounded from below, one can consider the following functional in the space of properly
immersed surfaces:
\begin{equation*}
  I(\Sigma)=\frac{1}{2} \inf R^M |\Sigma|+ \inf H^{\partial M} |\partial\Sigma|,
\end{equation*}
\noindent where $|\Sigma|$ denotes the area of $\Sigma$ and $|\partial \Sigma|$ denotes the length of $\partial \Sigma$. \\
\indent The next proposition gives an upper bound to $I(\Sigma)$ when one assumes that $\Sigma$
is a free boundary stable minimal surface:
\begin{proposition} \label{riginf}
  Let $(M,g)$ be a Riemannian three-manifold with boundary $\partial M$. Assume $R^M$ and $H^{\partial M}$ are bounded from below.
  If $\Sigma$ is a properly immersed, two-sided, free boundary stable minimal surface, then
  \begin{equation} \label{stabineq}
    I(\Sigma)\leq 2\pi\chi(\Sigma).
  \end{equation}
  where $\chi(\Sigma)$ is the Euler characteristic of $\Sigma$. Moreover, the equality holds if, and only if, 
  $\Sigma$ satisfies the following properties: \\
  \indent a) $\Sigma$ is totally geodesic in $M$ and $\partial \Sigma$ consists of geodesics of $\partial M$; \\
  \indent b) The scalar curvature $R^M$ is constant along $\Sigma$ and equal to $\inf R^M$, and the mean curvature $H^{\partial M}$ is constant along $\partial \Sigma$
  and equal to $\inf H^{\partial M}$;\\
  \indent c) $Ric(N,N)=0$, and $N$ is in the kernel of the shape operator of $\partial M$ along $\partial \Sigma$, where $N$ is the 
  unit normal vector field of $\Sigma$. \\
  \indent In particular, $a)$, $b)$ and $c)$ imply that $\Sigma$ has constant Gaussian curvature $\inf R^M/2$ and 
  $\partial \Sigma$ has constant geodesic curvature $\inf H^{\partial M}$ in $\Sigma$.
\end{proposition}
\indent Inequality (\ref{stabineq}) relates the scalar curvature of $M$, the mean curvature of $\partial M$ and the topology
of the free boundary stable $\Sigma$, as in Schoen and Yau's Theorem \ref{SchYau1}. This connection has also been studied by Chen,
Fraser and Pang \cite{CFP}. \\
\indent For further reference, we will call \textit{infinitesimally rigid} any properly embedded, two-sided,
free boundary surface $\Sigma$ in $M$ that satisfies properties $a)$, $b)$ and $c$). \\
\indent It is interesting to have in mind the following model situation. In Riemannian three-manifolds of the form
($\mathbb{R}\times \Sigma ,dt^2+g_0$), where $(\Sigma,g_0)$ is a compact Riemannian surface
with constant Gaussian curvature whose boundary has constant geodesic curvature, all the slices $\{t\}\times\Sigma$ satisfy the
hypotheses of Proposition \ref{riginf} and are infinitesimally rigid. They also have two additional properties: they are
in fact area-minimizing and each connected component of their boundary has the shortest possible length in its homotopy class inside 
the boundary of $\mathbb{R}\times\Sigma$. \\
\indent Given an infinitesimally rigid surface $\Sigma_0$, we construct a foliation $\{\Sigma_{t}\}_{t\in I}$
around $\Sigma_{0}$ by constant mean curvature free boundary surfaces and then analyze the behavior of the area of the
surfaces $\Sigma_t$ following the unified approach of \cite{MM}. When $\inf H^{\partial M}>0$
and each component of $\partial \Sigma$ is locally length-minimizing, or when $\inf H^{\partial M}=0$,
we prove that $|\Sigma_0|\geq|\Sigma_t|$ for every $t\in I$ (maybe for some smaller interval $I$). 
As a consequence, we obtain a local rigidity theorem for area-minimizing free boundary surfaces
in Riemannian three-manifolds with mean convex boundary (i.e., $H^{\partial M}\geq 0$):
\begin{theorem} \label{mainA}
  Let $(M, g)$ be a Riemannian three-manifold with mean convex boundary. Assume that $R^M$ is bounded from below. \\
  \indent Let $\Sigma$ be a properly embedded, two-sided, locally area-minimizing free boundary surface such that $I(\Sigma)=2\pi\chi(\Sigma)$. Assume
  that one of the following hypotheses holds: \\
  \indent i) each component of $\partial \Sigma$ is locally length-minimizing in $\partial M$; or \\
  \indent ii) $\inf H^{\partial M}=0$. \\
  \indent Then there exists a neighborhood of $\Sigma$ in $(M,g)$ that is isometric to
  $((-\epsilon,\epsilon)\times \Sigma, dt^2+g_{\Sigma})$, where $(\Sigma,g_{\Sigma})$ has cons\-tant Gaussian curvature 
  $\frac{1}{2}\inf R^{M}$ and $\partial \Sigma$ has constant geodesic curvature $\inf H^{\partial M}$ in $\Sigma$.
\end{theorem}
\indent We use this local result to prove some global rigidity theorems. \\
\indent Let ${\mathcal{F}_M}$ be the set of all immersed disks in $M$ whose boundaries are curves in $\partial M$ that 
are homotopically non-trivial in $\partial M$. If ${\mathcal{F}_M}$ is non-empty, we define
\begin{eqnarray*}
{\mathcal{A}}(M,g)=\inf_{\Sigma\in {\mathcal{F}_M}} |\Sigma| & \text{and}  &  {\mathcal{L}}(M,g)=\inf_{\Sigma\in {\mathcal{F}_M}} |\partial \Sigma|. 
\end{eqnarray*}
\indent Our first global rigidity theorem involves a combination of these geometric invariants.
\begin{theorem} \label{mainB}
  Let $(M,g)$ be a compact Riemannian three-manifold with mean convex boun\-da\-ry. Assume that
  ${\mathcal{F}}_M$ is non-empty. Then
  \begin{equation} \label{firstinvineq}
    \frac{1}{2}\inf R^M {\mathcal{A}}(M,g) + \inf H^{\partial M} {\mathcal{L}}(M,g) \leq 2\pi.
  \end{equation}
  \noindent Moreover, if equality holds, then the universal covering of $(M,g)$ is isometric to 
  $(\mathbb{R}\times \Sigma_0, dt^2+g_{0})$, where $(\Sigma_0,g_{0})$ is a disk with cons\-tant Gaussian curvature 
  $\inf R^{M}/2$ and $\partial \Sigma_0$ has constant geodesic curvature $\inf H^{\partial M}$ in $(\Sigma_0,g_{0})$.  
\end{theorem} 
\indent The case $\inf R^M=0$ and $\inf H^{\partial M}>0$, which includes in particular mean convex domains of the Euclidean space,
was treated by M. Li (see his preprint \cite{L}). His approach is similar to the one in \cite{BBN}. \\
\indent Our proof relies on the fact that ${\mathcal{A}}(M,g)$ can be realized as the area of a properly embedded free boundary minimal disk
$\Sigma_0$, by a classical result of Meeks and Yau \cite{MY1}. Since $H^{\partial M}\geq 0$, we can compare the invariant and $I(\Sigma_0)$, and hence inequality
(\ref{firstinvineq}) follows from Proposition \ref{riginf}.
When equality holds, $\Sigma_0$ must be infinitesimally rigid,
and then we use the local splitting around $\Sigma_0$ given by Theorem \ref{mainA}
and a standard continuation argument to obtain the global splitting of the universal covering. \\
\indent When $\inf R^{M}$ is negative, we also prove a rigidity theorem for solutions of the Plateau problem, which is
an immediate consequence of Theorem \ref{mainB}. \\
\indent As before, assume that $(M,g)$ is a compact Riemannian three-manifold with mean convex boundary. Another 
classical result of Meeks and Yau \cite{MY2} says that the Plateau problem has a properly embedded solution in $M$
for any given closed embedded curve in $\partial M$ that bounds a disk. \\
\indent In particular, by considering solutions of the Plateau problem for homotopically non-trivial curves in $\partial M$
that bound disks and have the shortest possible length among such curves, we prove the following
\begin{theorem} \label{mainC}
  Let $(M,g)$ be a compact Riemannian three-manifold with mean convex boun\-da\-ry such that $\inf R^{M} = -2$.
  Assume that ${\mathcal{F}}_M$ is non-empty. \\
  \indent If $\hat{\Sigma}$ is a solution to the Plateau problem for a homotopically non-trivial embedded curve in $\partial M$
  that bounds a disk and has length $\mathcal{L}(M,g)$, then
  \begin{equation}
    |\hat\Sigma| \geq \inf H^{\partial M} \mathcal{L}(M,g) - 2\pi.
  \end{equation}
  \noindent Moreover, if equality holds in $(3)$ for some $\hat{\Sigma}$, then the universal covering of $(M,g)$ is iso\-metric to 
  $(\mathbb{R}\times \Sigma_0, dt^2+g_{0})$, where $(\Sigma_0,g_{0})$ is a disk with cons\-tant Gaussian curvature 
  $-1$ and $\partial \Sigma_0$ has constant geodesic curvature $\inf H^{\partial M}$ in $\Sigma_0$.  
\end{theorem}

\medskip

\indent \textit{Acknowledgments.} I am grateful to my Ph.D advisor at IMPA, Fernando Cod\'a Marques, for his constant advice
and encouragement. I also thank Ivaldo Nunes for enlightening discussions about free boundary surfaces. Finally, I am grateful 
to the hospitality of the Institut Henri Poincar\'e, where the first drafts of this work were written 
in October/November 2012. I was supported by CNPq-Brazil and FAPERJ.

\section{Infinitesimal rigidity}

\indent Inequality (\ref{stabineq}) follows from the second variation formula of area for free boundary minimal surfaces,
  the Gauss equation and the Gauss-Bonnet theorem.
  \begin{proof}[Proof of Proposition \ref{riginf}]
    Let $\Sigma$ be a properly immersed, two-sided, free boundary stable minimal surface. Since $\Sigma$ is two-sided,
    there exists a unit vector field $N$ along $\Sigma$ that is normal to $\Sigma$. Let $X$ be the unit vector field on $\partial M$
    that is normal to $\partial M$ and points outside $M$. Since $\Sigma$ is free boundary, the unit conormal $\nu$ of $\partial \Sigma$
    that points outside $\Sigma$ coincides with $X$ along $\partial \Sigma$. \\
    \indent Recall that $H^{\partial M}$ is the trace of the shape operator $\nabla X$, under our convention.
    The free boundary hypothesis implies that $k$, the geodesic
    curvature of $\partial\Sigma$ in $\Sigma$, can be computed as $k=g(T,\nabla_{T}\nu)=g(T,\nabla_{T}X)$, where $T$ is a
    unit vector field tangent to $\partial \Sigma$. In particular, 
    \begin{equation} \label{Hek}
      H^{\partial M}= k + g(N,\nabla_{N}X).
    \end{equation}
    \indent The free boundary stability hypothesis means that, for every $\phi\in C^{\infty}(\Sigma)$,
    \begin{equation*}
      Q(\phi,\phi)= \int_\Sigma |\nabla \phi|^2 - (Ric(N,N) + |B|^2)\phi^2 dA - \int_{\partial\Sigma} g(N,\nabla_{N}X)\phi^2 dL\geq 0,
    \end{equation*}
    \noindent where $B$ denotes the second fundamental form of $\Sigma$. $Q(\phi,\phi)$ is the second variation of area 
    for variations with variational vector field $\phi N$ along $\Sigma$ (for the general second variation formula,
    see \cite{S}). \\
    \indent By evaluating $Q$ on the constant function $1$, we have the inequalities
    \begin{eqnarray*}
      0 & \geq & \int_\Sigma (Ric(N,N) + |B|^2)dA + \int_{\partial\Sigma} g(N,\nabla_{N}X)dL \\
	&   =  & \frac{1}{2} \int_\Sigma (R^M + H^2 + |B|^2)dA - \int_\Sigma K dA - \int_{\partial \Sigma} k dL + \int_{\partial \Sigma} H^{\partial M} dL \\
	& \geq & \frac{1}{2} \inf R^M |\Sigma| + \inf H^{\partial M}|\partial\Sigma| -2\pi\chi(\Sigma).
    \end{eqnarray*}
    \noindent where we used the Gauss equation, equation (\ref{Hek}) and the Gauss-Bonnet theorem. This proves inequality
    (\ref{stabineq}). \\
    \indent When the equality holds in (\ref{stabineq}), every inequality above is in fact an equality. One immediately sees
    that $\Sigma$ must be totally geodesic, $b)$ holds and $Q(1,1)=0$. By elementary considerations about bilinear forms, $Q(1,1)=0$ and
    $Q(\phi,\phi)\geq 0$ for every $\phi\in C^{\infty}(\Sigma)$ implies $Q(1,\phi)=0$ for every $\phi\in$ $C^{\infty}(\Sigma)$.
    Hence, by choosing appropriately the arbitrary test function $\phi$, we conclude that $Ric(N,N)=0$ and $g(N,\nabla_{N}X)=0$. \\
    \indent Since $\Sigma$ is totally geodesic, $\nabla_T T$ and $\nabla_{T} X=\nabla_{T}\nu$ are tangent to $\Sigma$.
    Hence, the geodesic curvature of $\partial \Sigma$ in $\partial M$ given by $g(N,\nabla_{T} T)$ vanishes, and since
    $\nabla_{T} X$ is also orthogonal to $X$ we conclude that $\nabla_{T} X$ is proportional to $T$, which means that $T$ and
    therefore $N$ are eigenvectors of $\nabla X$ on $\partial \Sigma$.
    The second parts of $a)$ and $c)$ follows. \\
    \indent The final statement is just a consequence of the Gauss equation and equation (\ref{Hek}). The converse is immediate from
    the Gauss-Bonnet theorem.
  \end{proof}

\section{Construction of the foliation}

\indent Given a properly embedded infinitesimally rigid surface $\Sigma$ in $M$, there are smooth vector fields $Z$ on $M$ such that
$Z(p)=N(p)$ $\forall p\in \Sigma$ and $Z(p)\in T_{p}\partial M$ $\forall p\in \partial M$. We fix 
$\phi=\phi(x,t)$ the flow of one of these vector fields and $\alpha$ a real number between zero and one. \\
\indent The next proposition gives a family of constant mean curvature free boundary surfaces around an infinitesimally rigid surface.
\begin{proposition} \label{foliation}
  Let $(M,g)$ be a Riemannian three-manifold with boundary $\partial M$. Assume $R^M$ and $H^{\partial M}$ are bounded from below.
  Let $\Sigma$ be a properly embedded, two-sided, free boundary surface. \\
  \indent If $\Sigma$ is infinitesimally rigid, then there exists $\epsilon >0$ and a function
  $w: \Sigma \times (-\epsilon,\epsilon)\rightarrow \mathbb{R}$ such that, for every  $t\in(-\epsilon,\epsilon)$,
  the set
  \begin{equation*}
  \Sigma_t = \{\phi(x,w(x,t));\, x\in\Sigma\}
  \end{equation*}
  \noindent is a free boundary surface with constant mean curvature $H(t)$. Moreover, for every $x\in\Sigma$ and every $t\in(-\epsilon,\epsilon)$,
  \begin{equation*}
    w(x,0)=0, \quad  \int_{\Sigma} \left( w(x,t)-t \right) dA = 0 \quad \text{and} \quad \frac{\partial}{\partial t}w(x,t)\Big|_{t=0} = 1.
  \end{equation*}
  \indent In particular, for some smaller $\epsilon$, $\{\Sigma_t\}_{t\in(-\epsilon,\epsilon)}$ is a foliation of a neighborhood
  of $\Sigma_0=\Sigma$ in $M$.
\end{proposition}
\begin{proof}
  As in the proof of Proposition \ref{riginf}, let $N$ denote the unit normal vector field of $\Sigma$,
  and let  $X$ denote the unit normal vector field of $\partial M$ that coincides with the exterior conormal $\nu$ of $\partial \Sigma$.
  Let $dA$ be the area element of $\Sigma$ and let $dL$ be the length element
  of $\partial \Sigma$. \\
  \indent Given a function $u$ in the H\"older space $C^{2,\alpha}(\Sigma)$, $0 < \alpha < 1$,
  we consider $\Sigma_{u}=\{\phi(x,u(x)); x\in\Sigma\}$, which is a properly embedded surface if the norm of $u$ is small 
  enough. We use the subscript $u$ to denote the quantities associated to $\Sigma_u$.
  For example, $H_{u}$ will denote the mean curvature of $\Sigma_{u}$, $N_u$ will denote
  the unit normal vector field of $\Sigma_u$ and $X_{u}$ will denote the restriction
  of $X$ to $\partial \Sigma_{u}$. In particular, 
  $\Sigma_{0}=\Sigma$, $H_{0}=0$ (since $\Sigma$ is totally geodesic) and $g(N_0,X_0)=0$ (since $\Sigma_0$ is free boundary). \\
  \indent Consider the Banach spaces $E=\{u\in C^{2,\alpha}; \int_{\Sigma}u dA=0\}$ and $F=\{u\in C^{0,\alpha}; \int_{\Sigma}u dA=0\}$. 
  Given small $\delta>0$ and $\epsilon >0$, we can define the map 
  $\Phi: (-\epsilon,\epsilon)\times (B(0,\delta)\subset E) \rightarrow F\times C^{1,\alpha}(\partial\Sigma)$ given by
  \begin{equation*}
    \Phi(t,u)=( \, \, H_{t+u}-\frac{1}{|\Sigma|}\int_{\Sigma}{H_{t+u}dA}, \,\, g(N_{t+u},X_{t+u}) \,\,).
  \end{equation*}
  
  \indent We claim that $D\Phi_{(0,0)}$ is an isomorphism when restricted to $0\times E$. \\  
  \indent In fact, for each $v\in E$, the map $f: (x,s)\in\Sigma\times(-\epsilon,\epsilon)\mapsto\phi(x,sv(x))\in M$ gives 
  a variation with variational vector field $\frac{\partial f}{\partial s}|_{s=0}=vZ=vN$ on $\Sigma$.
  Since $\Sigma$ is infinitesimally rigid we obtain (see Proposition \ref{normalvar} in the Appendix):
  \begin{equation*}
    D\Phi_{(0,0)}(0,v)= \frac{d}{ds}\Big|_{s=0}\Phi(0,sv)=(-\Delta_{\Sigma}v+\frac{1}{|\Sigma|}\int_{\partial \Sigma}\frac{\partial v}{\partial \nu}dL, \,\, -\frac{\partial v}{\partial \nu}).
  \end{equation*}
  \indent The claim follows from classical results for Neumann type boundary conditions for the Laplace operator (see for example \cite{LU}, page 137). \\
  \indent Now we apply the implicit function theorem: for some smaller $\epsilon$, there exists a function $t\in(-\epsilon,\epsilon)\mapsto u(t)\in B(0,\delta)\subset E$
  such that $u(0)=0$ and  $\Phi(t,u(t))=\Phi(0,0)=(0,0)$ for every $t$. In other words, the surfaces
  \begin{equation*}
    \Sigma_{t+u(t)}=\{\phi(x,t+u(t)(x));\,x\in\Sigma\}
  \end{equation*}
  \noindent are free boundary constant mean curvature surfaces. \\
  \indent Let $w:(x,t) \in \Sigma\times (-\epsilon,\epsilon) \mapsto t+u(t)(x) \in \mathbb{R}$. By definition,
  $w(x,0)=u(0)(x)=0$ for every $x\in\Sigma$ and $w(-,t)-t=u(t)$ belongs to $B(0,\delta)\subset E$ for every $t\in(-\epsilon,\epsilon)$. 
  Observe that the map $G: (x,s)\in\Sigma\times(-\epsilon,\epsilon)\mapsto\phi(x,w(x,s))\in M$ gives a variation of $\Sigma$
  with variational vector field on $\Sigma$ given by $\left(\frac{\partial w}{\partial t}|_{t=0}\right)N$.
  Since for every $t$ we have
  \begin{equation*}
    0=\Phi(t,u(t))=( \,\, H_{w(-,t)} - \frac{1}{|\Sigma|}\int_{\Sigma}H_{w(-,t)}dA \,\, , \,\, g(N_{w(-,t)},X_{w(-,t)}) \,\,),
  \end{equation*}
  \noindent by taking the derivative at $t=0$ we conclude that $\frac{\partial w}{\partial t}|_{t=0}$
  satisfies the homogeneous Neumann problem. Therefore it must be constant on $\Sigma$. Since
  $\int_{\Sigma} \left(w(x,t)-t \right) dA$ $= \int_{\Sigma} u(t)(x) dA =0$ for every $t$, by taking again
  a derivative at $t=0$ we conclude that $\int_{\Sigma}\left(\frac{\partial w}{\partial t}|_{t=0}\right)dA =|\Sigma|$.
  Hence, $\frac{\partial w}{\partial t}|_{t=0} = 1$, as claimed. \\
  \indent Since $G_0(x)=\phi(x,0)=x$, $\partial_{t} G (x,0)$ $=\frac{\partial w}{\partial t}|_{t=0}N_{0}(x)$ 
  $=N_{0}(x)$ for every $x$ in $\Sigma_0$ and $\Sigma_0$ is properly embedded, by taking a smaller $\epsilon$,
  if necessary, we can assume that $G$ parametrizes a foliation of $M$ around $\Sigma_0$. This finishes the proof
  of the proposition.
\end{proof}

\section{Local rigidity}

\indent We consider a Riemannian three-manifold with mean convex boundary and scalar curvature bounded
from below. First we analyze the behavior of the area of surfaces in the family constructed in section 3.
This analysis is based on \cite{MM}.
\begin{proposition} \label{anarea}
  Let $(M,g)$ be a Riemannian three-manifold with mean convex boundary and scalar curvature bounded from below.
  Let $\Sigma_0$ be a properly embedded, two-sided, free boundary, infinitesimally rigid surface.\\
  \indent Assume that one of the following hypotheses holds: \\
  \indent i) each component of $\partial \Sigma_0$ is locally length-minimizing in $\partial M$; or \\
  \indent ii) $\inf H^{\partial M}=0$. \\
  \indent Let $\{\Sigma_t\}_{t\in(-\epsilon,\epsilon)}$ be as in Proposition \ref{foliation}.
  Then $|\Sigma_0|\geq|\Sigma_{t}|$ for every $t\in(-\epsilon,\epsilon)$ (maybe for some smaller $\epsilon$).
\end{proposition}
\begin{proof}
  \indent Following the notation of Proposition \ref{foliation}, 
  let $G: \Sigma_0\times(-\epsilon,\epsilon)\rightarrow M$ given by $G_t(x)=\phi(x,w(x,t))$ parametrize
  the foliation $\{\Sigma_{t}\}_{t\in(-\epsilon,\epsilon)}$ around the infinitesimally rigid $\Sigma_0$. 
  After this point, we will use the subscript $t$ to denote the quantities associated to $\Sigma_{t}=G_{t}(\Sigma_0)$. \\
  \indent For each $t\in(-\epsilon,\epsilon)$, the lapse function on $\Sigma_t$ given by $\rho_{t}=g(\partial_{t}G,N_{t})$
  satisfies the equations (see Proposition \ref{cmcfreefol} in the Appendix):
  \begin{eqnarray}
  -H'(t) & = &\Delta_{t}\rho_{t}+(Ric(N_{t},N_{t})+|B_{t}|^2)\rho_{t}, \\
  \frac{\partial \rho_{t}}{\partial {\nu_{t}}}& = & g(N_{t},\nabla_{N_{t}}X)\rho_{t}.
  \end{eqnarray}
  \indent Furthermore, $\rho_{0}=1$, since $\partial_t G(x,0)=N_{0}(x)$ for every $x\in \Sigma$. Hence, we can assume
  $\rho_{t} >0$ for all $t\in(-\epsilon,\epsilon)$. From equation $(5)$ we have
  \begin{equation*}
    H'(t)\frac{1}{\rho_{t}}=-(\Delta_{t}\rho_{t})\frac{1}{\rho_{t}} - (Ric(N_{t},N_{t})+|B_{t}|^2).
  \end{equation*}
  \indent Using the Gauss equation, we rewrite
  \begin{equation*}
    H'(t)\frac{1}{\rho_{t}}=-(\Delta_{t}\rho_{t})\frac{1}{\rho_{t}} + K_{t} - \frac{1}{2}(R^{M}_{t}+H(t) ^2+|B_{t}|^2).
  \end{equation*}
  \indent Recalling that $H(t)$ is constant on $\Sigma_t$, we integrate by parts using equation (6) in order to get
  \begin{multline*}
   H'(t)\int _{\Sigma}\frac{1}{\rho_{t}}dA_{t}=-\int_{\Sigma}\frac{|\nabla_{t} p_{t}|^2}{\rho_{t}^2}dA_{t} - \int_{\partial \Sigma}g(N_{t},\nabla_{N_{t}} X)dL_{t} \\ + \int_{\Sigma}K_{t}dA_{t} - \frac{1}{2}\int_{\Sigma}(R^{M}_{t}+H(t) ^2+|B_{t}|^2)dA_{t}.
  \end{multline*}
  \indent Since each $\Sigma_t$ is free boundary, equation (\ref{Hek}) and the Gauss-Bonnet theorem imply
  \begin{multline*}
   H'(t)\int _{\Sigma}\frac{1}{\rho_{t}}dA_{t}=-\int_{\Sigma}\frac{|\nabla_{t} p_{t}|^2}{\rho_{t}^2}dA_{t} - \frac{1}{2}\int_{\Sigma}(R^{M}_{t}+H(t) ^2+|B_{t}|^2)dA_{t} \\ - \int_{\partial \Sigma}H_{t}^{\partial M}dL_{t} + 2\pi\chi(\Sigma_0).
  \end{multline*}
  \indent Finally, since $\Sigma_0$ is infinitesimally rigid, the Gauss-Bonnet theorem implies that $I(\Sigma_0)=2\pi\chi(\Sigma_0)$.
  Hence, we have the following inequality:
  \begin{eqnarray*}
    H'(t)\int _{\Sigma}\frac{1}{\rho_{t}}dA_{t} & \leq & I(\Sigma_0)-I(\Sigma_t) \\
                                              & = & \frac{1}{2}\inf{R^M}(|\Sigma_0|-|\Sigma_t|) + \inf H^{\partial M}(|\partial \Sigma_0|- |\partial \Sigma_t|).     
  \end{eqnarray*}
  \indent By hypothesis, $\inf H^{\partial M} \geq 0$. If each boundary component is locally length-minimizing, the second term in the right hand side is
  less than or equal to zero, and in case $\inf H^{\partial M}=0$, it is obviously zero. Therefore
  \begin{equation*}
    H'(t)\int _{\Sigma}\frac{1}{\rho_{t}}dA_{t} \leq \frac{1}{2}\inf{R^M}(|\Sigma_0|-|\Sigma_t|)= -\frac{1}{2}\inf{R^M}\int_{0}^{t} \frac{d}{ds}|\Sigma_s| ds.
  \end{equation*}
  \indent Since each $\Sigma_t$ is free boundary, the first variation formula of area gives
  \begin{equation} \label{firstvar}
    \frac{d}{dt}|\Sigma_t|= \int_\Sigma \rho_t H(t)dA_t= H(t)\int_\Sigma \rho_t dA_t.
  \end{equation}
  \indent Therefore
  \begin{equation} \label{fundineq}
    H'(t)\int _{\Sigma}\frac{1}{\rho_{t}}dA_{t} \leq -\frac{1}{2}\inf{R^M}\int_{0}^{t} H(s) \left(\int_\Sigma \rho_s dA_s \right)  ds.
  \end{equation}
  
  \medskip 
  
  \noindent \textbf{Claim}: there exists $\epsilon>0$ such that $H(t)\leq 0$ for every $t\in[0,\epsilon)$. \\
  
  \indent We consider three cases: \\
  
  \indent a) $\inf R^{M}=0$. \\
  \indent Then it follows immediately from (\ref{fundineq}) that $H'(t)\leq 0$ for every $t\in[0,\epsilon)$. Since $H(0)=0$,
  the claim follows. \\
  
  \indent b) $\inf R^{M}>0$. \\
  \indent Let $\varphi(t) = \int_{\Sigma} \frac{1}{\rho_t} dA_t$ and $\xi(t) = \int_{\Sigma} \rho_t dA_t$.
  Inequality (\ref{fundineq}) can be rewritten as
  \begin{equation} \label{fundineqB}
    H'(t) \leq -\frac{1}{2}\inf R^{M} \frac{1}{\varphi(t)}\int_{0}^{t}H(s)\xi(s) ds.
  \end{equation}
  \indent By continuity, we can assume that there exists a constant $C>0$ such that
  $\frac{1}{\varphi(t)}\int_{0}^{t}\xi(s)ds\leq 2C$ for every $t\in[0,\epsilon]$. \\
  \indent Choose $\epsilon>0$ such that $C\inf R^{M}\epsilon < 1$. Then $H(t)\leq 0$ for every $t\in[0,\epsilon)$.
  In fact, suppose that there exists $t_{+} \in (0, \epsilon)$ such that $H(t_{+}) > 0$. By continuity, there exists $t_{-} \in [0, t_{+}]$ 
  such that $H(t)\geq H(t_{-})$ for every
  $t\in [0, t_{+}]$. Notice that $H(t_{-}) \leq H(0) = 0$. By the mean value theorem, there exists
  $t_{1} \in (t_{-}, t_{+})$ such that $H(t_{+})-H(t_{-}) = H'(t_{1})(t_{+}-t_{-})$. Hence, 
  since $\inf R^{M} >0$, inequality (\ref{fundineqB}) gives
  \begin{eqnarray*}
    \frac{H(t_{+})-H(t_{-})}{t_{+}-t_{-}} & = H'(t_1) & \leq \frac{1}{2}\inf R^{M}\frac{1}{\varphi(t_{1})}\int_{0}^{t_{1}}(-H(s))\xi(s)ds \\
                                          &           & \leq \frac{1}{2}\inf R^{M}(-H(t_{-}))\left(\frac{1}{\varphi(t_{1})}\int_{0}^{t_1}\xi(s)ds\right) \\
                                          &           & \leq \inf R^{M}(-H(t_{-}))C.
  \end{eqnarray*}
  \indent It follows that $H(t_{+}) \leq H(t_{-})(1-C\inf R^{M}\epsilon)$, which is a contradiction
  since $H(t_{+}) > 0$ and $H(t_{-} ) \leq {0}$. \\
  
  \indent c) $\inf R^{M}<0$. \\
  \indent Choose $\epsilon>0$ such that $-C\inf R^{M}\epsilon < 1$, where $C>0$ is the same constant that appears in
  case b). Then $H(t)\leq 0$ for every $t\in[0,\epsilon)$. In fact, suppose that there exists $t_{0} \in (0, \epsilon)$
  such that $H(t_{0}) > 0$. Let
  \begin{equation*}
    R = \{t \in [0,t_0]; \, H(t) \geq H(t_0)\}.
  \end{equation*}
  \indent Let $t^{*}\in[0,\epsilon]$ be the infimum of $R$. Observe that, by the definition of $t^{*}$,
  $H(t)\leq H(t_{0})=H(t^{*})$ for every $t\in[0,t^{*}]$. \\
  \indent If $t^{*}>0$, then the mean value theorem implies that there exists $t_{1} \in (0,t^{*})$ 
  such that $H(t^{*}) = H'(t_{1})t^{*}$, since $H(0) = 0$.
  Hence, since $\inf R^{M} < 0$, inequality (\ref{fundineqB}) gives
  \begin{eqnarray*}
    \frac{H(t^{*})}{t^{*}} & = H'(t_1) & \leq -\frac{1}{2}\inf R^{M}\frac{1}{\varphi(t_1)}\int_{0}^{t_1}H(s)\xi(s)ds \\
                           &           & \leq -\frac{1}{2}\inf R^{M} H(t^{*})\left(\frac{1}{\varphi(t_1)} \int_{0}^{t_1}\xi(s)ds\right) \\
                           &           & \leq -\inf R^{M}H(t^{*})C.
  \end{eqnarray*}           
  \indent It follows that $H(t^{*})(1+C\inf R^{M} H(t^{*})\epsilon)\leq 0$. This is a contradiction since $H(t^{*})$ $=H(t_0)>0$. \\
  \indent Hence, $t^{*}=0$, which is again a contradiction since $0=H(0)\geq H(t_0) >0$. \\
  
  \indent This proves the claim. By equation (\ref{firstvar}), we conclude that $|\Sigma_0|\geq |\Sigma_t|$
  for every $t\in[0,\epsilon)$. The proof that $|\Sigma_0|\geq |\Sigma_t|$ for every $t\in(-\epsilon,0]$ is analogous. \\
\end{proof}
\indent We are now ready to prove the local splitting result, Theorem \ref{mainA}.

\begin{proof}[Proof of Theorem \ref{mainA}]
  \indent Since $\Sigma$ is locally area-minimizing and $I(\Sigma)=2\pi\chi(\Sigma)$, $\Sigma$ is infinitesimally
  rigid. From propositions \ref{foliation} and \ref{anarea} we obtain a foliation $\{\Sigma_t\}_{t\in(-\epsilon,\epsilon)}$
  around $\Sigma_0=\Sigma$ such that $|\Sigma_t|\leq|\Sigma_0|$ for every $t\in(-\epsilon,\epsilon)$. Since $\Sigma$ is locally
  area-minimizing, each $\Sigma_t$ is also locally area-minimizing, with $|\Sigma_t|=|\Sigma_0|$. \\
  \indent It is immediate to see that when $\inf H^{\partial M}=0$ or 
  when the components of $\partial \Sigma_{0}$ are locally length-minimizing,
  \begin{equation*}
    2\pi=I(\Sigma_0) \leq I(\Sigma_t) \leq 2\pi,
  \end{equation*}
  \noindent which implies that each $\Sigma_t$ is infinitesimally rigid. From equations (5) and (6) in Proposition \ref{anarea}, one
  sees that for each $t$ the lapse function $\rho_t$ satisfies the homogeneous Neumann problem.
  Therefore $\rho_t$ is a constant function on $\Sigma_t$. \\
  \indent Since we have a foliation, the normal fields of $\Sigma_{t}$ define locally a vector field on $M$. This field is parallel
  (see \cite{BBN}, \cite{MM} or \cite{N}). In particular, its flow is a flow by isometries and therefore provides the local splitting:
  a neighborhood of $\Sigma_0$ is in fact isometric to the product
  $((-\epsilon,\epsilon)\times \Sigma_{0},dt^2+g_{\Sigma_0})$. Since $\Sigma_0$ is infinitesimally rigid, 
  $(\Sigma_{0},g_{\Sigma_0})$ has constant Gaussian curvature $\inf R^M/2$ and $\partial \Sigma_0$ has constant
  geodesic curvature $\inf H^{M}$ in $\Sigma_0$. \\
\end{proof}

\section{Global rigidity}

\indent Before we begin the proofs, we state precisely the result of Meeks and Yau about the
existence of area-minimizing free boundary disks that we will use in the sequel (see \cite{MY1}).
\begin{theorem}[Meeks and Yau] \label{MeeksYau}
  Let $(M,g)$ be a compact Riemannian three-manifold with mean convex boundary. If
  ${\mathcal{F}_M}$ is non-empty, then \\
  \indent 1) There exists an immersed minimal disk $\Sigma_0$ in $M$ such that $\partial \Sigma_0$ represents a
  homotopically non-trivial curve on $\partial M$ and $|\Sigma_0|=\mathcal{A}(M,g)$ . \\
  \indent 2) Any such least area immersed disk is in fact a properly embedded free boundary disk.
\end{theorem}
\indent We are now ready to prove our main theorems.
\begin{proof}[Proof of Theorem \ref{mainB}]
  Since ${\mathcal{F}}_M$ is non-empty, Theorem \ref{MeeksYau} says that there exists a properly embedded free boundary
  minimal disk $\Sigma_0\in {\mathcal{F}}_M$ such that $|\Sigma_0| = {\mathcal{A}}(M,g)$. Since $\Sigma_0$ is two-sided
  and free boundary stable, the inequality follows from Proposition \ref{riginf}:
  \begin{equation*}
    \frac{1}{2}\inf R^M {\mathcal{A}}(M,g) + \inf H^{\partial M} {\mathcal{L}}(M,g) \leq I(\Sigma_0) \leq 2\pi.
  \end{equation*}
  \indent Assume that the equality holds. In case $\inf H^{\partial M}$ is not zero, $\partial \Sigma_0$ must have length ${\mathcal{L}}(M,g)$,
  hence it is length-minimizing. In any case, we can apply Theorem \ref{mainA} to get a local splitting of $(M,g)$
  around $\Sigma_0$. \\
  \indent Let $\exp$ denote the exponential map of $(M,g)$. Let $S$ be the set all $t>0$ such that
  the map $\Psi: [-t,t]\times\Sigma_0\rightarrow M$ given by $\Psi(s,x)=\exp_{x}(sN_0(x))$ is well-defined,
  $\Psi([-t,t]\times\partial\Sigma_{0})$ is contained in $\partial M$ and 
  $\Psi: ((-t,t)\times\Sigma_0, ds^2+g_{\Sigma_0})\rightarrow (M,g)$ is a local isometry. \\
  \indent $S$ is non-empty because of the local splitting. Standard arguments imply that $S=[0,+\infty)$. 
  Therefore we have a well-defined local isometry 
  \begin{equation*}
    \Psi: (t,x) \in (\mathbb{R}\times\Sigma_0,dt^2+g_{\Sigma_0})\mapsto \exp_{x}(tN_0(x))\in (M,g),
  \end{equation*}
  \noindent such that $\Psi(\mathbb{R}\times\partial\Sigma_{0})$ is contained in $\partial M$. Such
  $\Psi$ is a covering map. This finishes the proof of Theorem \ref{mainB}.
\end{proof}

\medskip

\indent In order to prove Theorem \ref{mainC}, consider any $\hat{\Sigma}$ as in its statement.
$\hat{\Sigma}$ has area at least ${\mathcal{A}}(M,g)$ and $\partial \hat{\Sigma}$ has length ${\mathcal{L}}(M,g)$.
When $\inf R^{M}$ is negative,
\begin{equation*}
  I(\hat{\Sigma}) = \frac{1}{2}\inf R^M |\hat{\Sigma}| + \inf H^{\partial M} |\partial \hat{\Sigma}| \leq \frac{1}{2}\inf R^M {\mathcal{A}}(M,g) + \inf H^{\partial M} {\mathcal{L}}(M,g).
\end{equation*}
\noindent and therefore Theorem \ref{mainC} is an immediate corollary of Theorem \ref{mainB}.

\section*{Appendix}

\indent For completeness we include some general formulae for the infinitesimal variation of some geometric quantities of properly
immersed hypersurfaces under variations of the ambient manifold $(M^{n+1},g)$ that leave the boundary of the hypersurface inside $\partial M$. \\
\indent We begin by fixing some notations. Let $(M^{n+1},g)$ be a Riemannian manifold with boundary $\partial M$.
Let $X$ denote the unit normal vector field along $\partial M$ that points outside $\partial M$. \\
\indent Let $\Sigma^n$ be a manifold with boundary $\partial \Sigma$ and assume
$\Sigma$ is immersed in $M$ in such way that $\partial \Sigma$ is contained in $\partial M$.  The unit conormal of
$\partial \Sigma$ that points outside $\Sigma$ will be denoted by $\nu$.
Given $N$ a local unit normal vector field to $\Sigma$, the second fundamental form is the symmetric tensor $B$ on $\Sigma$
given by $B(U,W)=g(\nabla_{U} N, W)$ for every $U$, $W$ tangent to $\Sigma$. The mean curvature $H$ is the trace of $B$. $\Sigma$ is called \textit{minimal} when
$H=0$ on $\Sigma$ and \textit{free boundary} when $\nu=X$ on $\partial \Sigma$. \\
\indent We consider variations of $\Sigma$ given by smooth maps $f: \Sigma\times(-\epsilon,\epsilon)\rightarrow M$ such that,
for every $t\in(-\epsilon,\epsilon)$, the map $f_{t}: x\in \Sigma\mapsto f(x,t)\in M$ is an immersion of $\Sigma$ in $M$ such
that $f_{t}(\partial \Sigma)$ is contained in $\partial M$. \\
\indent The subscript $t$ will be used to denote quantities associated to $\Sigma_t=f_{t}(\Sigma)$. For example,
$N_t$ will denote a local unit vector field normal to $\Sigma_t$ and $H_t$ will
denote the mean curvature of $\Sigma_t$. \\
\indent It will be useful for the computations to introduce local coordinates $x^1, \ldots, x^n$ in $\Sigma$. We will
also use the simplified notation 
\begin{eqnarray*}
  \partial_{t} = \frac{\partial f}{\partial t} & \text{and} & \partial_{i} = \frac{\partial f}{\partial x_i} ,
\end{eqnarray*}
\noindent where $i$ runs from $1$ to $n$. $\partial_t$ is called the variational vector field. We decompose it in its tangent and normal components:
\begin{equation*}
  \partial_{t}=\partial_{t}^T + v_t N_t,
\end{equation*}
\noindent where $v_t$ is the function on $\Sigma_t$ defined by $v_t=g(\partial_t,N_t)$. \\
\indent First we look at the variation of the metric tensor $g_{ij}=g(\partial_i,\partial_j)$:
\begin{proposition}
  \begin{eqnarray*}
    \partial_t g_{ij} & = & g(\nabla_{\partial_i} \partial_{t},\partial_{j}) + g(\partial_i,\nabla_{\partial_{j}} \partial_t), \\
    \partial_t g^{ij} & = & -2g^{ik}g^{jl}g(\nabla_{\partial_k} \partial_{t},\partial_l).
  \end{eqnarray*}
\end{proposition}
\begin{proof}
  The first equation is straightforward. The second follows from differentiating $g^{ik}g_{kl}=\delta_{il}$. 
\end{proof}
\indent From the well-known formula for the derivative of the determinant, $$(\det U)'=\det (U) \text{tr}(U'),$$ we deduce:
\begin{proposition}
  The first variation of area is given by 
  \begin{equation*}
    \frac{d}{dt}|\Sigma_t| = \int_{\Sigma} H_t v_t dA_t + \int_{\partial \Sigma} g(\nu_t,\frac{\partial f}{\partial t})dL_{t}. 
  \end{equation*}
\end{proposition}
\begin{proof}
  Observe that
  \begin{eqnarray*}
    \partial_{t}\sqrt{\det[g_{ij}]} & = & \frac{1}{2}g^{ij}\partial_{t}g_{ij}\sqrt{\det[g_{ij}]} \\
                                    & = & g^{ij}g(\nabla_{\partial_i} \partial_{t},\partial_{j})\sqrt{\det[g_{ij}]} \\
                                    & = & (g^{ij}g(\nabla_{\partial_i} \partial_{t}^T,\partial_{j}) + g^{ij}g(\nabla_{\partial_i} N_t,\partial_{j})v_t)\sqrt{\det[g_{ij}]} \\
                                    & = & (\text{div}_{\Sigma_t}\partial_t^T + H_{t}v_t)\sqrt{\det[g_{ij}]}.
  \end{eqnarray*}
  \indent The first variation formula of area follows.
\end{proof}
\indent Next we look at the variations of the normal field.
\begin{proposition} \label{varnormal}
  \begin{eqnarray*}
    \nabla_{\partial_{i}} N_t & = & g^{kl}B_{il}\partial_{k} , \\
    \nabla_{\partial_{t}} N_t & = & \nabla_{(\partial_{t})^T} N_t - \nabla^{\Sigma_t} v_t.
  \end{eqnarray*}
  \noindent where $\nabla^{\Sigma} v_t$ is the gradient of the function $v_t$ on $\Sigma_t$.
\end{proposition}
\begin{proof}
  Since $g(N_t,N_t)=1$, $\nabla_{\partial_i} N_t$ and $\nabla_{\partial_t}N_t$ are tangent to $\Sigma_t$. The first equation is just the expression
  of $\partial_{i} N_t$ in the basis $\{\partial_{k}\}$. On the other hand, since $g(N_t,\partial_{i})=0$, we have
  \begin{equation*}
    \nabla_{\partial_{t}} N_{t} = g^{ik}g(\nabla_{\partial_{t}} N_{t} , \partial_{k})\partial_{i} = -g^{ik}g(N_t, \nabla_{\partial_{t}} \partial_{k})\partial_{i} = -g^{ik}g(N_t,\nabla_{\partial_{k}}\partial_{t})\partial_{i}.
  \end{equation*}
  \indent In local coordinates, the gradient of $v_t$ in $\Sigma_t$ is given by $\nabla^{\Sigma_t} v_t = (g^{ij}\partial_{j} v_t)\partial_{i}$. Then we have
  \begin{equation*}
    g^{ik}g(N_t,\nabla_{\partial_{k}} (v_tN_t))\partial_{i} = (g^{ik}\partial_{k} v_{t})\partial_{i} = \nabla^{\Sigma_t} v_t.
  \end{equation*}
  \indent Therefore
  \begin{equation*}
    \nabla_{\partial_{t}} N_t = \nabla_{(\partial_{t})^T} N_t - \nabla^{\Sigma_t} v_t.
  \end{equation*}
\end{proof} 
\indent Before we compute the variation of the mean curvature, let us recall the \textit{Codazzi equation}:
\begin{equation*}
  g(R(U,V)N_t,W) = (\nabla^{\Sigma_t}_{U}B)(V,W) - (\nabla^{\Sigma_t}_{V}B)(U,W).
\end{equation*}
\indent In this equation, $R$ denotes the Riemann curvature tensor of $(M,g)$ and $U$, $V$ and $W$ are tangent to $\Sigma_t$. \\
\indent Taking $U=\partial_i$, $W=\partial_k$ and contracting, we obtain
\begin{equation*}
  Ric(V,N_t) =  g^{ik}(\nabla^{\Sigma_t}_{\partial_{i}}B)(V,\partial_{k}) - dH_t(V).
\end{equation*}
\noindent for every $V$ tangent to $\Sigma_t$.
\begin{proposition} \label{meancurv}
  The variation of the mean curvature is given by
  \begin{equation*}
    \partial_t H_t = dH_t (\partial_{t}^{T}) - L_{\Sigma_t}v_t.
  \end{equation*}
  \noindent where $L_{\Sigma_t} = \Delta_{\Sigma_t} + Ric(N_t,N_t) + |B_t|^2$ is the Jacobi operator.
\end{proposition}
\begin{proof}
  Since $H_t= g^{ij}g(\nabla_{\partial_{i}} N_t, \partial_{j})$,  
  \begin{eqnarray*}
    \partial_t H_t & = & \partial_t g^{ij}g(\nabla_{\partial_{i}} N_t, \partial_{j}) + g^{ij}g(\nabla_{\partial_t}\nabla_{\partial_{i}} N_t,\partial_{j}) + g^{ij}g(\nabla_{\partial_i} N_t, \nabla_{\partial_t} \partial_{j}) \\
                   & = & -2g^{ik}g^{jl}g(\nabla_{\partial_{k}}\partial_{t},\partial_{l})g(\nabla_{\partial_{i}} N_t, \partial_{j}) + g^{ij}g(R(\partial_{t},\partial_{i})N_t,\partial_{j}) \\
                   &   & + g^{ij}g(\nabla_{\partial_i} \nabla_{\partial_t} N_t, \partial_{j}) + g^{ij}g(\nabla_{\partial_{i}} N_t, \nabla_{\partial_{j}}\partial_{t}) \\
                   & = & -2g^{ik}g(\nabla_{\partial_k} \partial_{t},\nabla_{\partial_{i}} N_{t}) - Ric(\partial_t,N_t) \\
                   &   & + g^{ij}g(\nabla_{\partial_i} \nabla_{\partial_t} N_t, \partial_{j}) + g^{ij}g(\nabla_{\partial_{i}} N_t, \nabla_{\partial_{j}}\partial_{t}) \\
                   & = & -g^{ij}g(\nabla_{\partial_{i}} N_t, \nabla_{\partial_{j}}\partial_{t}) - Ric(\partial_{t},N_t) \\
                   &   & +g^{ij}g(\nabla_{\partial_i} (\nabla_{\partial_t^T} N_t), \partial_{j}) - g^{ij}g(\nabla_{\partial_i}(\nabla^{\Sigma_t} v) ,\partial_{j}).
  \end{eqnarray*}
  \indent Now we use the contracted Codazzi equation:
  \begin{eqnarray*}
    Ric(\partial_t ^{T}, N_t) & = & g^{ij}(\nabla^{\Sigma_t}_{\partial_{i}}B)(\partial_{t}^T,\partial_{j}) -dH(\partial_t ^T)\\
                              & = & g^{ij}\partial_{i} g(\nabla_{\partial_{t}^T} N_t,\partial_{j}) - g^{ij}g(\nabla_{(\nabla_{\partial_{i}}\partial_t^T)^{T} } N_t,\partial_{j}) \\
                              &   & -g^{ij}g(\nabla_{\partial_t ^T} N_t,(\nabla_{\partial_{i}} {\partial_{j}})^T) - dH(\partial_t ^T) \\
                              & = & g^{ij}(\partial_{i} g(\nabla_{\partial_{t}^T} N_t,\partial_{j}) - g(\nabla_{\partial_t ^T} N_t,\nabla_{\partial_{i}} {\partial_{j}})) \\
                              &   & -g^{ij}g(\nabla_{\partial_{j}}N_t,(\nabla_{\partial_{i}}\partial_{t}^T)^T) - dH(\partial_t ^T)  \\
                              & = & g^{ij}g(\nabla_{\partial_{i}} (\nabla_{\partial_{t}^T} N_t) , \partial_{j}) - g^{ij}g(\nabla_{\partial_{j}} N_t, \nabla_{\partial_{i}}\partial_{t}^T)-dH(\partial_{t}^T).
  \end{eqnarray*}
  \indent Hence, canceling out the corresponding terms, we have
  \begin{eqnarray*}
    \partial_t H_t & = & - g^{ij}g(\nabla_{\partial_{i}} N_t, \nabla_{\partial_{j}}N_t)v_t - Ric(N_t,N_t)v_t \\
                   &   & + dH(\partial_t ^T) - g^{ij}g(\nabla_{\partial_i}(\nabla^{\Sigma_t} v_t) ,\partial_{j}).
  \end{eqnarray*}
  \indent The formula follows.
  \end{proof}
\indent Finally, we specialize the formulae above in the two particular cases we used in this paper. The proofs are immediate.
\begin{proposition} \label{normalvar}
  If $\Sigma_0$ is free boundary and $(\partial_{t})^T=0$ at $t=0$, then 
  \begin{eqnarray*}
     (\partial_t H_t)|_{t=0}  = -L_{\Sigma_0}v_0 & \text{and} & \partial_{t}g(N_t,X)|_{t=0} = -\frac{\partial v_0}{\partial \nu_0} + g(N_0,\nabla_{N_0} X)v_0.
  \end{eqnarray*}
\end{proposition}
\begin{proposition} \label{cmcfreefol}
  If each $\Sigma_t$ is a constant mean curvature free boundary surface, then
  \begin{eqnarray*}
    \partial_t H_t  = - L_{\Sigma_t}v_{t} & \text{and} & \frac{\partial v_{t}}{\partial {\nu_{t}}} = g(N_{t},\nabla_{N_{t}}X)v_{t}.
  \end{eqnarray*}
\end{proposition}

\bibliographystyle{amsbook}

\end{document}